\documentclass[11pt]{article}
\usepackage[english]{babel}
\usepackage{amssymb}
\usepackage{amsmath}
\usepackage{amsthm}
\usepackage[a4paper,margin=2.54cm]{geometry}
\usepackage{todonotes}
\usepackage{ytableau}

\usepackage[hyperindex,breaklinks]{hyperref}
\PassOptionsToPackage{linktocpage}{hyperref}

\theoremstyle{plain}
\newtheorem*{theorem*}{Theorem}
\newtheorem{theorem}{Theorem}[section] 
\newtheorem{lemma}[theorem]{Lemma}
\newtheorem{proposition}[theorem]{Proposition}

\theoremstyle{definition}

\newtheorem{remark}[theorem]{Remark}

\makeatletter
\@addtoreset{equation}{section}
\makeatother
\numberwithin{equation}{section}

\DeclareMathOperator{\Wr}{Wr}
\DeclareMathOperator{\He}{He}
\DeclareMathOperator{\Id}{Id}
\DeclareMathOperator{\sgn}{sgn}
\DeclareMathOperator{\htt}{ht}

\usepackage{authblk}
\title{Recurrence relations for Wronskian Laguerre polynomials}

\author[1]{Niels Bonneux}
\author[2]{Marco Stevens}
\affil[1,2]{KU Leuven, Department of Mathematics, 
	
	Celestijnenlaan~200B box 2400, 3001 Leuven, Belgium. 
	
	E-mail:~{\tt niels.bonneux@kuleuven.be} and {\tt marco.stevens@kuleuven.be}
}
\date{\today}                     
\usepackage{hyperref}
\hypersetup{colorlinks, citecolor=black, linkcolor=black}
\AtBeginDocument{} 

\begin{document}
\maketitle

\begin{abstract}
The 3-term recurrence relation for Hermite polynomials was recently generalized to a recurrence relation for Wronskians of Hermite polynomials. In this note, a similar generalization for Laguerre polynomials is obtained.
\end{abstract}

\section{Introduction}
Hermite and Laguerre polynomials are well-studied classical orthogonal polynomials and can be defined via their 3-term recurrence relation~\cite{Szego}. Applying the Wronskian operator to a finite set of these polynomials yields Wronskian Hermite and Wronskian Laguerre polynomials. They appear in rational solutions of Painlev\'e equations, see for example~\cite{Clarkson-survey,VanAssche} and the references therein, and play a key role in the theory of exceptional orthogonal polynomials~\cite{Bonneux_Kuijlaars,Duran-Laguerre,GomezUllate_Grandati_Milson}.

A recurrence relation for Wronskian Hermite polynomials was derived in~\cite{Bonneux_Stevens} by direct computation of determinants. It generalizes the classical 3-term recurrence relation of Hermite polynomials. Subsequently, these results were shown to hold in general for Wronskians of Appell polynomials in~\cite{Bonneux_Hamaker_Stembridge_Stevens}, by using a connection with the theory of symmetric functions.

Modified Laguerre polynomials satisfy the Appell property and hence specifying the results in~\cite{Bonneux_Hamaker_Stembridge_Stevens} to this case yields a recurrence relation for Wronskian Laguerre polynomials. Remarkably, this relation is not a generalization of the 3-term recurrence relation for Laguerre polynomials. However, following the route of direct computations of determinants that was used in~\cite{Bonneux_Stevens}, a different recurrence relation is obtained and this is a proper generalization of the 3-term recurrence relation. It expresses Wronskian Laguerre polynomials of degree $n$ in terms of Wronskian Laguerre polynomials of degrees $n-1$ and $n-2$. This note is dedicated to this new recurrence relation, which is stated in Theorem~\ref{thm:WLrecurrence2}, and cannot be obtained from symmetric function theory as in~\cite{Bonneux_Hamaker_Stembridge_Stevens}.

The rest of this section discusses the conventions of this note: Laguerre polynomials, the basics of the theory of partitions~\cite{Macdonald,Stanley_EC2}, the definition of Wronskian Laguerre polynomials, the main result Theorem~\ref{thm:WLrecurrence2}, a comparison with the result for exceptional Laguerre polynomials and with the result for Wronskian Hermite polynomials in~\cite{Bonneux_Stevens}, and a remark why the main result does not extend to Wronskian Jacobi polynomials. Subsequently, the proof of Theorem~\ref{thm:WLrecurrence2} is given in Section~\ref{sec:proof}. The last section contains an alternative proof of the averaging property of Wronskian Laguerre polynomials with respect to the Plancherel measure. In contrast with the proof given in~\cite{Bonneux_Hamaker_Stembridge_Stevens}, Section~\ref{sec:average} does not make use of the theory of symmetric functions. Instead, it is based on an inductive argument using Theorem~\ref{thm:WLrecurrence2}.

\subsection{Laguerre polynomials}
Laguerre polynomials with parameter $\alpha$ can be defined by their 3-term recurrence relation
\begin{equation*}
	L_{n}^{(\alpha)}(x)
	= \left(2+\frac{-x+\alpha-1}{n}\right)L_{n-1}^{(\alpha)}(x) - \left(1+\frac{\alpha-1}{n}\right)L_{n-2}^{(\alpha)}(x)
\end{equation*}
for $n\geq 2$, together with the initial conditions $L_{0}^{(\alpha)}(x)=1$ and $L_{1}^{(\alpha)}(x)=-x+\alpha+1$, see formula~(5.1.10) in~\cite{Szego}. Using elementary identities for Laguerre polynomials, i.e., formulas~(5.1.13) and~(5.1.14) from~\cite{Szego}, it is easy to show that Laguerre polynomials satisfy
\begin{equation}\label{eq:Lrecurrencebis}
	n L_{n}^{(\alpha)}(x)
	= (-x+\alpha+1) L_{n-1}^{(\alpha+1)}(x) - x \, L_{n-2}^{(\alpha+2)}(x)
\end{equation} 
for all $n\geq 2$. Therefore setting
\begin{equation*}
	l_{n}^{(\alpha)}(x)
	:= n! \, L_{n}^{(\alpha-n)}(-x)
\end{equation*}
yields that~\eqref{eq:Lrecurrencebis} can be written as
\begin{equation}\label{eq:Lrecurrence2}
	l_{n}^{(\alpha)}(x)
	= (x+\alpha-n+1) \, l_{n-1}^{(\alpha)}(x)  + x \, (n-1) \,l_{n-2}^{(\alpha)}(x)
\end{equation}
for all $n\geq 2$, together with $l_{0}^{(\alpha)}(x)=1$ and $l_{1}^{(\alpha)}(x)=x+\alpha$. Moreover, $l_{n}^{(\alpha)}$ is a monic polynomial of degree $n$ and the sequence $(l_n^{(\alpha)})_{n=0}^\infty$ satisfies the Appell property, that~is
\begin{equation}\label{eq:LAppell}
	\frac{d}{dx}l_{n}^{(\alpha)}(x)
	= n \, l_{n-1}^{(\alpha)}(x)
\end{equation}
for all $n\geq1$. It is noteworthy that even though the transformation from the sequence of polynomials $L_n^{(\alpha)}$ to $\ell_n^{(\alpha)}$ is a modification of the Laguerre polynomials, it are precisely these modified Laguerre polynomials that appear in the rational solutions of the third and fifth Painlev\'e equation, see~\cite{Clarkson-survey,VanAssche} and the references therein.

\subsection{Partitions and degree vectors}
Wronskian Laguerre polynomials are, similarly as in~\cite{Bonneux_Hamaker_Stembridge_Stevens,Bonneux_Kuijlaars}, labelled by (integer) partitions. A brief overview of the necessary notions of the theory of partitions is given below. For a slightly more extended discussion (using the same notation), the reader is referred to~\cite{Bonneux_Hamaker_Stembridge_Stevens}, and for a thorough introduction to~\cite{Macdonald,Stanley_EC2}.

A partition is a finite sequence $\lambda=(\lambda_1,\lambda_2,\dots,\lambda_r)$ of strictly positive integers such that $\lambda_1\geq \lambda_2 \geq \dots \geq \lambda_r>0$. The length of~$\lambda$ is $\ell(\lambda)=r$ and the size is $\lvert \lambda \rvert = \sum_{i} \lambda_i$. If $\lvert \lambda \rvert = n$, then $\lambda \vdash n$. Each partition can be identified with its Young diagram $D_{\lambda}$, which consists of~$r$ rows of boxes, and row $i$ contains $\lambda_i$ boxes. The set of all partitions $\mathbb{Y}$ is partially ordered by $\mu \leq \lambda$ if the Young diagram of~$\mu$ fits within the Young diagram of~$\lambda$, or equivalently, if $\mu_i \leq \lambda_i$ for all $i=1,2,\dots,\ell(\mu)$. This partial ordering turns the set $\mathbb{Y}$ into a lattice, called the Young lattice.

If $\mu\leq\lambda$, then the difference of Young diagrams $D_{\lambda}\setminus D_{\mu}$ defines the skew partition $\lambda / \mu$, and $\lvert \lambda /\mu \rvert := \lvert \lambda \rvert - \lvert \mu \rvert$. In the case that $\lambda /\mu$ is a border strip of size $k$~\cite[I.1]{Macdonald}, that is, it is connected and does not contain a $2\times 2$-square, then $\mu \in \mathcal{R}_{k}^-(\lambda)$, and $\htt(\lambda/\mu)$ denotes the number of rows minus 1. Furthermore, $\lambda \in \mathcal{R}_{k}^+(\mu)$ if and only if $\mu \in \mathcal{R}_{k}^-(\lambda)$. In the special case that $\lvert \lambda /\mu \rvert =1$, the notation $\mu\lessdot\lambda$ (or equivalently $\lambda\gtrdot\mu$) is used. In this case, the Young diagram of~$\mu$ is obtained by removing precisely one box from the Young diagram of~$\lambda$. If this box is in the $i^\textrm{th}$ row, then
\begin{equation}
	\label{eq:defcontent}
	c(\lambda/\mu)=\lambda_i-i
\end{equation}
denotes the content of the skew partition. This is in correspondence with the general notion of contents of partitions~\cite{Macdonald,Stanley_EC2}. If $\lambda / \mu$ is a border strip of size 2, then the Young diagram of~$\mu$ is obtained from the Young diagram of~$\lambda$ by removing a domino tile. If this domino tile is horizontally placed, then $\htt(\lambda/\mu)=0$ and if it is vertically placed, then $\htt(\lambda/\mu)=1$. This also corresponds to the assignments of signs in~\cite{Bonneux_Stevens}.

Any partition $\lambda=(\lambda_1,\lambda_2,\dots,\lambda_r)$ has a degree vector ${n_\lambda=(n_1,n_2,\dots,n_r)}$ associated to it, defined by $n_i=\lambda_i+r-i$ for all $i=1,2,\dots,r$. Note that a partition is a weakly decreasing sequence while its degree vector is strictly decreasing. 

The number of directed paths in the Young lattice from $\emptyset$ (the unique partition of size~$0$) to a partition $\lambda$, denoted by $F_{\lambda}$, can be written in terms of the degree vector~as
\begin{equation}
	\label{eq:Flambda}
	F_\lambda= \frac{\lvert \lambda \rvert! \prod_{i<j} (n_i-n_j)}{\prod_{i} n_i!}
\end{equation}
and is equal to the number of standard Young tableaux of shape $\lambda$~\cite[Corollary~3.31]{Baik_Deift_Suidan}, as well as to the dimension of the irreducible representation of the symmetric group associated to $\lambda$.

\subsection{Wronskian Laguerre polynomials}
Following the convention of~\cite{Bonneux_Hamaker_Stembridge_Stevens} and~\cite{Bonneux_Stevens}, the Wronskian Laguerre polynomial with parameter $\alpha$ associated to the partition $\lambda=(\lambda_1,\lambda_2,\dots,\lambda_r)$ is defined by
\begin{equation}
	\label{eq:WLP}
	l_{\lambda}^{(\alpha)} = \frac{\Wr[l_{n_1}^{(\alpha)}, l_{n_2}^{(\alpha)}, \dots, l_{n_r}^{(\alpha)}]}{\prod_{i<j} (n_j-n_i)}
\end{equation}
where the (modified) Laguerre polynomials $l_n^{(\alpha)}$ are defined by~\eqref{eq:Lrecurrence2} and ${n_\lambda=(n_1,n_2,\dots,n_r)}$ is the degree vector of~$\lambda$. From the definition it immediately follows that $l_\lambda^{(\alpha)}$ is a monic polynomial of degree $\lvert \lambda \rvert$. Taking the trivial partition $\lambda=(n)$ yields $l_{(n)}^{(\alpha)}=l_n^{(\alpha)}$, and so the Wronskian Laguerre polynomials are a generalization of the Laguerre polynomials. The main result of this note, which is proven in Section~\ref{sec:proof}, is the following recurrence relation for the Wronskian polynomials.
\begin{theorem}\label{thm:WLrecurrence2}
	If $\lambda$ is a non-empty partition, then
	\begin{equation}\label{eq:WLrecurrence2}
		F_{\lambda} l_{\lambda}^{(\alpha)}(x)
		= \sum_{\mu \lessdot \lambda} (x+\alpha-c(\lambda/\mu)) F_{\mu} l_{\mu}^{(\alpha)}(x)
		\, \,  +  \, \,  x (|\lambda|-1) \sum_{\rho \in \mathcal{R}_2^{-}(\lambda)}(-1)^{\htt(\lambda/\rho)} F_{\rho} l_{\rho}^{(\alpha)}(x)
	\end{equation}
	where $c(\lambda/\mu)$ is the content defined in~\eqref{eq:defcontent}.
\end{theorem}
Taking $\lambda=(n)$ in~\eqref{eq:WLrecurrence2} precisely yields the 3-term recurrence~\eqref{eq:Lrecurrence2}. This is in contrast with the recurrence relation obtained in~\cite{Bonneux_Hamaker_Stembridge_Stevens} for Wronskian Laguerre polynomials, which says
\begin{multline}\label{eq:WLrecurrence}
	F_\lambda l^{(\alpha)}_\lambda(x)
	= (x+\alpha) \sum_{\mu \lessdot \lambda} F_\mu l^{(\alpha)}_\mu(x) 
	+ \alpha \sum_{k=2}^{|\lambda|} (-1)^{k-1}\frac{(|\lambda|-1)!}{(|\lambda|-k)!} \sum_{\nu\in\mathcal{R}_k^{-}(\lambda)} (-1)^{\htt(\lambda/\nu)} F_\nu l^{(\alpha)}_\nu(x)
\end{multline}
for any partition $|\lambda|\geq1$. Taking $\lambda=(n)$ in this relation yields
\begin{equation*}
	l_n^{(\alpha)}(x)
	= (x+\alpha)l_{n-1}^{(\alpha)}(x)+ \alpha \sum_{k=2}^n (-1)^{k-1} \frac{(n-1)!}{(n-k)!} l_{n-k}^{(\alpha)}(x)
\end{equation*}
which is fundamentally different in shape from the 3-term recurrence~\eqref{eq:Lrecurrence2}, but nevertheless equivalent with it by using formulas (5.1.13) and (5.1.14) from~\cite{Szego}. It is noteworthy that there are less terms in the right-hand side of~\eqref{eq:WLrecurrence2} than in~\eqref{eq:WLrecurrence}.

\begin{remark}
	Both recurrence relations~\eqref{eq:WLrecurrence2} and~\eqref{eq:WLrecurrence} generate the set of Wronskian Laguerre polynomials together with the initial conditions $l_{\emptyset}^{(\alpha)}(x)=1$ and ${l_{(1)}^{(\alpha)}(x)=x+\alpha}$.
\end{remark}

\begin{remark}
Exceptional Laguerre polynomials~\cite{Bonneux_Kuijlaars,Duran-Laguerre} are constructed using Wronskians of classical Laguerre polynomials~$L_n^{(\alpha)}$ instead of the modified polynomials~$\ell_n^{(\alpha)}$. More concretely, they make use of the polynomials
\begin{equation}\label{eq:WL}
	L_{\lambda}^{(\alpha)} 
		= c_\lambda \Wr[L_{n_1}^{(\alpha)}, L_{n_2}^{(\alpha)}, \dots, L_{n_r}^{(\alpha)}]
\end{equation}
that are defined for every partition $\lambda$, with $c_\lambda$ being a normalization constant to obtain monic polynomials. Using modified polynomials (as in~\eqref{eq:WLP}) or classical Laguerre polynomials (as in~\eqref{eq:WL}) turns out to yield different Wronskian polynomials, to such an extent that the recurrence relations in Theorem~\ref{thm:WLrecurrence2} or in~\eqref{eq:WLrecurrence} do not translate to recurrence relations for exceptional Laguerre polynomials such as those obtained in~\cite[Corollary~5.1]{Duran-Recurrence}. Nevertheless, it is noteworthy that~\eqref{eq:WLP} and~\eqref{eq:WL} do relate for partitions whose Young diagrams are rectangles. More precisely, if $\lambda=(n,n,\dots,n)$ with $\ell(\lambda)=m$, then $L_{\lambda}^{(\alpha)}(x) = l_{\lambda'}^{(-\alpha-n)}(x)$ for any $\alpha$, where $\lambda'=(m,m,\dots,m)$ with $\ell(\lambda')=n$ is the conjugate partition of~$\lambda$.
\end{remark}

\begin{remark}
	The Hermite polynomials $(\He_n)_{n=0}^\infty$ are defined by
	\begin{equation}\label{eq:3termHermite}
		\He_n(x)
		=x\He_{n-1}(x) - (n-1) \He_{n-2}(x)
	\end{equation}
	for $n\geq 2$, together with $\He_0(x)=1$ and $\He_1(x)=x$, see~\cite[Chapter 9]{Jackson}. The Wronskian Hermite polynomials, defined by
	\begin{equation*}
		\He_\lambda = \frac{\Wr[\He_{n_1},\He_{n_2},\dots,\He_{n_r}]}{\prod_{i<j} (n_j-n_i)}
	\end{equation*} 
	satisfy
	\begin{equation}
		\label{eq:recurrenceWHP}
		F_\lambda \He_\lambda(x) = x \sum_{\mu \lessdot \lambda} F_\mu \He_\mu(x) - (\lvert \lambda \rvert -1) \sum_{\rho \in \mathcal{R}_2^-(\lambda)} (-1)^{\htt(\lambda/\rho)} F_\rho \He_\rho(x)
	\end{equation}
	for any non-empty partition $\lambda$, as was shown in~\cite{Bonneux_Stevens} by the direct computation of determinants. Notice the similarity between~\eqref{eq:WLrecurrence2} and~\eqref{eq:recurrenceWHP} and the fact that~\eqref{eq:recurrenceWHP} reduces to~\eqref{eq:3termHermite} when taking $\lambda=(n)$. The relation~\eqref{eq:recurrenceWHP} also follows by specifying the general recurrence relation for Wronskian Appell polynomials in~\cite{Bonneux_Hamaker_Stembridge_Stevens}, which can be done since the Hermite polynomials form an Appell sequence.
\end{remark}

\begin{remark}
	\label{rem:Jacobi}
	The Hermite and Laguerre polynomials are two types of classical orthogonal polynomials. The third (and last) type is formed by the class of Jacobi polynomials $P_n^{(\alpha,\beta)}$~\cite{Szego}, which (for fixed parameters $\alpha$ and $\beta$) do not satisfy the Appell property. However, assuming that $\alpha+\beta\not\in \mathbb{Z}_{\leq 0}$, the modified Jacobi polynomials
	\begin{equation*}
		A_n^{(\alpha,\beta)}(x)
		:= \frac{ 2^n n! }{(\alpha+\beta-n+1)_n} P^{(\alpha-n,\beta-n)}_n(x)
	\end{equation*} 
	where $(a)_n=a(a+1)\cdots(a+n-1)$ denotes the Pochhammer symbol, do satisfy the Appell property and their 3-term recurrence is 
	\begin{multline}\label{eq:Jrecurrence}
		(\alpha+\beta-n+1) A_{n}^{(\alpha,\beta)}(x)
		= \left((\alpha+\beta-2n+2)x+\alpha-\beta\right) A_{n-1}^{(\alpha,\beta)}(x) \\ + (x^2-1)(n-1) A_{n-2}^{(\alpha,\beta)}(x)
	\end{multline}
	for $n\geq 2$. Hence, by the general results in~\cite{Bonneux_Hamaker_Stembridge_Stevens}, there is a generating recurrence relation for the Wronskian Jacobi polynomials
	\[A_\lambda^{(\alpha,\beta)} = \frac{\Wr[A_{n_1}^{(\alpha,\beta)}, A_{n_2}^{(\alpha,\beta)}, \dots, A_{n_r}^{(\alpha,\beta)}]}{\prod_{i<j} (n_j-n_i)}\]
	but this relation is not of the form
	\begin{equation*}
		F_{\lambda} A_{\lambda}^{(\alpha,\beta)}(x)
		= \sum_{\mu \lessdot \lambda} (a_{\mu,\lambda} x + b_{\mu,\lambda}) F_{\mu} A_{\mu}^{(\alpha,\beta)}(x)
		+ \sum_{\rho \in \mathcal{R}_2^{-}(\lambda)} (c_{\rho,\lambda} x^2 + d_{\rho,\lambda} x + e_{\rho,\lambda}) F_{\rho} A_{\rho}^{(\alpha,\beta)}(x)	
	\end{equation*}
	where $a_{\mu,\lambda},b_{\mu,\lambda},c_{\rho,\lambda},d_{\rho,\lambda},e_{\rho,\lambda}\in\mathbb{R}$. In fact, an implementation in the computer software Maple shows that a recurrence relation of this form does not exist. The structure of the proof of the recurrence relation for Wronskian Hermite and Laguerre polynomials, as given in Section~\ref{sec:proof}, does not have an analogue for the Jacobi case. This is because both terms in the right-hand side of the 3-term recurrence relation~\eqref{eq:Jrecurrence} are polynomials of degree $n$, whereas there is only one such term in~\eqref{eq:Lrecurrence2} and~\eqref{eq:3termHermite}.
\end{remark}

\section{Proof of Theorem~\ref{thm:WLrecurrence2}}\label{sec:proof}
The proof of the recurrence relation~\eqref{eq:WLrecurrence2} for Wronskian Laguerre polynomials has the same structure and ideas as the proof of the recurrence relation for Wronskian Hermite polynomials which is given in~\cite[Section 5]{Bonneux_Stevens}. This structure has three parts. 
\begin{enumerate}
	\item The first part is a necessary rewriting exercise of the Wronskian polynomial and works for every Appell sequence.
	\item The rewriting exercise enables invoking the 3-term recurrence relation in the second part. Hence, the specifics of the second part depend on the sequence of polynomials.
	\item Invoking the 3-term recurrence relation yields a few terms. For the Hermite case, as treated in~\cite{Bonneux_Stevens}, there are two terms, which can be rewritten as a sum over the partitions $\mu \lessdot \lambda$ and $\rho \in \mathcal{R}_2^-(\lambda)$, respectively. In the Laguerre case, there are three terms; two of those are the analogous terms of the Hermite case, and the third one vanishes, as shown below. To build further upon Remark~\ref{rem:Jacobi}: completing part 1 and 2 for the Jacobi case yields terms which cannot be rewritten in any known combination of Wronskian Jacobi polynomials.
\end{enumerate}

\noindent \textbf{Part 1.}
The polynomial $F_\lambda l_\lambda^{(\alpha)}$ can be rewritten to make it suitable for invoking the 3-term recurrence relation in part 2. For this, fix a partition $\lambda=(\lambda_1,\lambda_2,\dots,\lambda_r)$ with degree vector $n_\lambda=(n_1,n_2,\dots,n_r)$. For any permutation $\sigma \in S_r$ and for any $i\in\{1,2,\dots,r\}$, write 
\begin{equation}
	\label{eq:sigmanlambda}
	\sigma(n_\lambda)_i = n_i-\sigma(i)+1
\end{equation} 
as in~\cite{Bonneux_Stevens}. Now, use~\eqref{eq:Flambda} for $F_\lambda$, evaluate the determinant in~\eqref{eq:WLP} as a sum over permutations, and use the Appell property~\eqref{eq:LAppell} repeatedly to arrive at
\begin{equation}\label{eq:proof1}
	F_{\lambda} l_{\lambda}^{(\alpha)}(x)=(-1)^{\frac{r(r-1)}{2}} |\lambda|! \sum_{\sigma\in S_r} \sgn(\sigma) \prod_{i=1}^{r} \frac{l_{\sigma(n_\lambda)_i}^{(\alpha)}(x)}{\sigma(n_\lambda)_i!}
\end{equation}
which is similar to~\cite[formula~(5.7)]{Bonneux_Stevens}. Subsequently, for $j=1,2,\dots,r$ write 
\begin{equation}
	\label{eq:nlambdaj}
	n_\lambda[j]=(n_1,\dots,n_{j-1},n_j-1,n_{j+1},\dots,n_r)
\end{equation}
as in~\cite{Bonneux_Stevens}. Then part 4 of Lemma 5.4 in~\cite{Bonneux_Stevens} says that
\begin{equation*}
	|\lambda|!  \prod_{i=1}^{r} \frac{1}{\sigma(n_\lambda)_i!}
	= \sum_{j=1}^{r}(|\lambda|-1)!  \prod_{i=1}^{r} \frac{1}{\sigma(n_\lambda[j])_i!}
\end{equation*}
which can be applied in~\eqref{eq:proof1}. It yields
\begin{equation}
	\label{eq:endpart1}
	F_{\lambda} l_{\lambda}^{(\alpha)}(x)
	= \sum_{j=1}^{r}(-1)^{\frac{r(r-1)}{2}}(|\lambda|-1)! \sum_{\sigma\in S_r} \sgn(\sigma) \prod_{i=1}^{r} \frac{l_{\sigma(n_\lambda)_i}^{(\alpha)}(x)}{\sigma(n_\lambda[j])_i!}
\end{equation}
which is the suitable form for invoking the 3-term recurrence relation in part~2. As mentioned before, so far, nothing depends on the specific case of Laguerre polynomials: any other Appell sequence satisfies the same equation. \\

\noindent \textbf{Part 2.}
The next step is to use the 3-term recurrence relation. More concretely, for each term in~\eqref{eq:endpart1}, the relation is applied to the $j^\textrm{th}$ factor in the product for every $\sigma \in S_r$. For this factor, the 3-term recurrence relation~\eqref{eq:Lrecurrence2} says
\begin{equation*}
	l_{\sigma(n_\lambda)_j}^{(\alpha)}(x)
	= (x+\alpha-\lambda_j+j+\sigma(j)-r) \, l_{\sigma(n_\lambda)_j-1}^{(\alpha)}(x)
	+ x(\sigma(n_\lambda)_j-1) \, l_{\sigma(n_\lambda)_j-2}^{(\alpha)}(x)
\end{equation*}
which follows from $\sigma(n_\lambda)_j=n_j-\sigma(j)+1$ and $n_j=\lambda_j+r-j$. Therefore
\begin{equation}\label{eq:proof3}
	F_{\lambda} l_{\lambda}^{(\alpha)}(x)
	= A + (-1)^{\frac{r(r-1)}{2}}(|\lambda|-1)!  B + C
\end{equation}
where
\begin{align}
	A	
	&= \sum_{j=1}^{r}(-1)^{\frac{r(r-1)}{2}}(|\lambda|-1)! \, (x+\alpha-\lambda_j+j) \sum_{\sigma\in S_r} \sgn(\sigma)  \frac{l^{(\alpha)}_{\sigma(n_\lambda)_j-1}(x) \prod_{i\neq j}l_{\sigma(n_\lambda)_i}^{(\alpha)}(x)}{\prod_{i}\sigma(n_\lambda[j])_i!} 
	\label{eq:proofA} \\
	B	
	&= \sum_{j=1}^{r} \sum_{\sigma\in S_r} \sgn(\sigma)  (\sigma(j)-r) \frac{l^{(\alpha)}_{\sigma(n_\lambda)_j-1}(x) \prod_{i\neq j}l_{\sigma(n_\lambda)_i}^{(\alpha)}(x)}{\prod_{i}\sigma(n_\lambda[j])_i!}
	\label{eq:proofB} \\
	C
	&= x\sum_{j=1}^{r}(-1)^{\frac{r(r-1)}{2}}(|\lambda|-1)!\sum_{\sigma\in S_r} \sgn(\sigma) (\sigma(n_\lambda)_j-1) \frac{l^{(\alpha)}_{\sigma(n_\lambda)_j-2}(x) \prod_{i\neq j}l_{\sigma(n_\lambda)_i}^{(\alpha)}(x)}{\prod_{i}\sigma(n_\lambda[j])_i!}	
	\label{eq:proofC}
\end{align}
and these terms can now be treated separately. \\

\noindent \textbf{Part 3.}
The decomposition~\eqref{eq:proof3} should be compared with Proposition 5.6 in~\cite{Bonneux_Stevens}: in the Hermite case analogues of the terms $A$ and $C$ exist, but not for the term $B$. In fact, completely similar as in the proof of~\cite[Theorem~3.1]{Bonneux_Stevens}, the terms $A$ and $C$, see~\eqref{eq:proofA} and~\eqref{eq:proofC}, are equal to
\begin{equation}\label{eq:proof4}
	A
	= \sum_{\mu \lessdot \lambda} (x+\alpha-c(\lambda/\mu)) F_{\mu} l_{\mu}^{(\alpha)}(x)
	\qquad \qquad
	C
	= x (|\lambda|-1) \sum_{\rho \in \mathcal{R}_2^{-}(\lambda)}(-1)^{\htt(\lambda/\rho)} F_{\rho} l_{\rho}^{(\alpha)}(x)
\end{equation}
and hence it is now sufficient to prove that $B=0$. For this, a symmetry argument that again does not depend on the specifics of the Laguerre polynomials suffices. Namely, define the set
\begin{equation*}
	X
	=\{(j,\sigma) \in \{1,2,\dots,r\}\times S_r \mid \sigma(j)\neq r \}
\end{equation*}
so that
\begin{equation}\label{eq:B2}
	B
	= \sum_{(j,\sigma)\in X} b_{j,\sigma}
\end{equation}
where $b_{j,\sigma}$ denotes the full expression of the term in the double sum in $B$, see~\eqref{eq:proofB}. Next, define the map
\begin{equation*}
	T: X \to X: (j,\sigma) \mapsto (k,\tau)
\end{equation*}
where $k=\sigma^{-1}(\sigma(j)+1)$ and $\tau= \sigma (j \, k)$. Then $\sgn(\sigma)=-\sgn(\tau)$ and $\sigma(j)=\tau(k)$, and therefore $T$ is well-defined. A direct computation yields $T\circ T=\Id_X$ so that $T$ is bijective, and moreover, $\sigma(n_\lambda[j])_i=\tau(n_\lambda[k])_i$ for all $i$, which follows directly from~\eqref{eq:sigmanlambda} and~\eqref{eq:nlambdaj}. Applying this to~\eqref{eq:B2} yields
\begin{equation*}
	2B
	= \sum_{(j,\sigma)\in X} \left( b_{j,\sigma} + b_{T(j,\sigma)} \right)
\end{equation*}
and so $B=0$ as all individual terms in the previous sum vanish because $b_{j,\sigma}=-b_{T(j,\sigma)}$ for all $(j,\sigma)\in X$. Hence, combining~\eqref{eq:proof3} and~\eqref{eq:proof4}, together with $B=0$, yields~\eqref{eq:WLrecurrence2}. \qed

\begin{remark} 
	In the above proof, the 3-term recurrence relation~\eqref{eq:Lrecurrence2} was invoked in part~2, to obtain the terms $A$, $B$ and $C$. A careful analysis of these terms in part~3 yields the recurrence relation~\eqref{eq:WLrecurrence2}. However, it is not necessary that the invoked recurrence relation is a 3-term recurrence; as long as in part~3 the resulting terms can be rewritten in terms of Wronskian polynomials of lower degrees, this technique yields a way to derive recurrence relations for Wronskian polynomials.
	
	It is in this way that the generating recurrence relation for Wronskian Appell polynomials, see~\cite[Theorem~6.3]{Bonneux_Hamaker_Stembridge_Stevens}, can be derived without using results of symmetric function theory. Namely, let $(A_n)_{n=0}^\infty$ be an Appell sequence, that is $A_0(x)=1$ and $A_n'(x)=nA_{n-1}(x)$ for all $n\geq1$. Then, the exponential generating function is of the form
	\begin{equation*}
		\sum_{n=0}^{\infty} A_n(x) \frac{t^n}{n!}
		= e^{xt} f_A(t)
	\end{equation*}
	for some formal power series $f_A$. Next, set 
	\begin{equation*}
		z_n
		:= A_n(0)
		\qquad \qquad
		\log(f_A(t))
		:= \sum_{n=1}^{\infty} c_n \frac{t^n}{n!}
	\end{equation*}
	such that
	\begin{equation}\label{eq:cnvszn}
		f_A(t)
		= \sum_{n=0}^{\infty} z_n \frac{t^n}{n!}
		\qquad \qquad
		z_n
		= c_n + \sum_{i=1}^{n-1} \binom{n-1}{i} c_{n-i} \, z_i
	\end{equation}
	see~\cite[Section 3.3]{Bonneux_Hamaker_Stembridge_Stevens}.
	
	\begin{lemma}\label{lem:ASrecurrence}
		Any Appell sequence $(A_n)_{n=0}^{\infty}$ can be generated by
		\begin{equation}\label{eq:ASrecurrence}
			A_n(x)
			= x A_{n-1}(x)
			+ \sum_{k=1}^{n} \binom{n-1}{k-1} c_k \, A_{n-k}(x)
		\end{equation}
		for $n\geq1$, along with the initial condition $A_0(x)=1$.
	\end{lemma} 
	\begin{proof}
		Setting $x=0$ in~\eqref{eq:ASrecurrence} yields, after some elementary rewriting, the expression given in~\eqref{eq:cnvszn}. Therefore, it is sufficient to show that the polynomials on both sides of the equality~\eqref{eq:ASrecurrence} have the same derivative. Approaching by induction on $n$ and using the Appell property, this yields an easy exercise.
	\end{proof} 
\end{remark}
Invoking the recurrence relation~\eqref{eq:ASrecurrence} in part 2 of the above structure gives $n+1$ terms, which can be rewritten to the terms that appear in~\cite[Theorem~6.3]{Bonneux_Hamaker_Stembridge_Stevens}. Since the details are similar as those in the proof of Theorem~3.1 in~\cite{Bonneux_Stevens}, they are left out.

\section{Average Wronskian Laguerre polynomial}\label{sec:average}
Wronskian Appell polynomials were constructed in~\cite{Bonneux_Hamaker_Stembridge_Stevens} for any Appell sequence. Since the sequence of (modified) Laguerre polynomials is an Appell sequence by~\eqref{eq:LAppell}, the results obtained there also hold specifically for Wronskian Laguerre polynomials as stated in~\cite[Section 7.3]{Bonneux_Hamaker_Stembridge_Stevens}. For example, for any $n\geq 0$,
\begin{equation}\label{eq:average}
	\sum_{\lambda\vdash n} 	\frac{F_{\lambda}^2}{n!} l_{\lambda}^{(\alpha)}(x)=(x+\alpha)^n
\end{equation}
which describes the average Wronskian Laguerre polynomial of degree $n$ with respect to the Plancherel measure~\cite[Definition 1.5]{Baik_Deift_Suidan}, and follows from the weighted average property for Schur functions~\cite[Corollary~7.12.5]{Stanley_EC2}. As shown below, this result can also be proven by induction on $n$ using the recurrence relation~\eqref{eq:WLrecurrence2}. This is analogous to the proof of the averaging result for Wronskian Hermite polynomials in~\cite[Theorem~3.4]{Bonneux_Stevens}.

\begin{proof}[Proof of~\eqref{eq:average}]
	By induction on $n:=|\lambda|$. For $n=0$ or $n=1$, the result is trivial because the average is taken of only 1 polynomial. Therefore, suppose that $n > 1$ and assume that the statement is true for $n-1$. Applying~\eqref{eq:WLrecurrence2} on the left-hand side of~\eqref{eq:average} yields
	\begin{multline}
		\label{eq:tripledoublesumaverage}
		\sum_{\lambda\vdash n} 	\frac{F_{\lambda}^2}{n!} l_{\lambda}^{(\alpha)}(x)
		= 
		\frac{x+\alpha}{n!} \sum_{\lambda\vdash n}\sum_{\mu \lessdot \lambda} F_{\lambda} F_{\mu} l_{\mu}^{(\alpha)}(x)
		-
		\frac{1}{n!}\sum_{\lambda\vdash n}\sum_{\mu \lessdot \lambda} c(\lambda/\mu) F_{\lambda} F_{\mu} l_{\mu}^{(\alpha)}(x)
		\\ +
		\frac{x (|\lambda|-1)}{n!} \sum_{\lambda\vdash n} \sum_{\rho \in \mathcal{R}_2^{-}(\lambda)}(-1)^{\htt(\lambda/\rho)} F_{\lambda} F_{\rho}  l_{\rho}^{(\alpha)}(x)
	\end{multline}
	where the first term of the recurrence relation~\eqref{eq:WLrecurrence2} is separated into two parts. Next, consider each double sum separately. Interchanging sums leads to the equalities
	\begin{align*}
		\sum_{\lambda\vdash n}\sum_{\mu \lessdot \lambda} F_{\lambda} F_{\mu} l_{\mu}^{(\alpha)}(x)
		& = \sum_{\mu\vdash n-1} F_{\mu} l_{\mu}^{(\alpha)}(x) \sum_{\lambda \gtrdot \mu} F_{\lambda} \\
		\sum_{\lambda\vdash n}\sum_{\mu \lessdot \lambda} c(\lambda/\mu) F_{\lambda} F_{\mu} l_{\mu}^{(\alpha)}(x)
		&= \sum_{\mu\vdash n-1} F_{\mu} l_{\mu}^{(\alpha)}(x) \sum_{\lambda \gtrdot \mu} F_{\lambda} \, c(\lambda/\mu)   \\
		\sum_{\lambda\vdash n} \sum_{\rho \in \mathcal{R}_2^{-}(\lambda)}(-1)^{\htt(\lambda/\rho)} F_{\rho} F_{\lambda} l_{\rho}^{(\alpha)}(x)
		&=  \sum_{\rho\vdash n-2} F_{\rho} l_{\rho}^{(\alpha)}(x) \sum_{\lambda \in \mathcal{R}_2^{+}(\rho)}(-1)^{\htt(\lambda/\rho)} F_{\lambda} 	
	\end{align*}
	where for each equality, the last sum can be calculated explicitly. Namely, Lemma 7.1 and Lemma 7.2 in~\cite{Bonneux_Stevens} state that
	\begin{equation*}
		\sum_{\lambda \gtrdot \mu} F_{\lambda}
		= (|\mu|+1) F_{\mu}
		\qquad 
		\qquad
		\sum_{\lambda \in \mathcal{R}_2^{+}(\rho)}(-1)^{\htt(\lambda/\rho)} F_{\lambda} 	
		= 0	
	\end{equation*}
	while Proposition~\ref{prop:technicalidentity} below shows that
	\begin{equation*}
		\sum_{\lambda \gtrdot \mu} F_{\lambda} \, c(\lambda/\mu)
		= 0
	\end{equation*}
	whence only the first double sum in~\eqref{eq:tripledoublesumaverage} does not vanish. So
	\begin{equation*}
		\sum_{\lambda\vdash n} 	\frac{F_{\lambda}^2}{n!} l_{\lambda}^{(\alpha)}(x)
		= (x+\alpha) \sum_{\mu\vdash n-1} \frac{F_{\mu}^2}{(n-1)!} l_{\mu}^{(\alpha)}(x)
		= (x+\alpha)^n
	\end{equation*}
	where the last equality is obtained by the induction hypothesis. This ends the proof.
\end{proof}

The following lemma is used in Proposition~\ref{prop:technicalidentity}.

\begin{lemma}\label{lem:technicalidentity}
	Let $\mu$ be a partition. Then
	\begin{equation}\label{eq:technicalidentity}
		\sum_{\lambda \gtrdot \mu} c(\lambda/\mu)
		= \sum_{\rho \lessdot \mu} c(\mu/\rho)
	\end{equation}
	where $c(\lambda /\mu)$ is defined in~\eqref{eq:defcontent}.
\end{lemma}
\begin{proof}
	Write $\mu=(\mu_1^{k_1},\mu_2^{k_2},\dots,\mu_m^{k_m})$, meaning that $\mu_t$ is repeated $k_t$ times for ${t=1,2,\dots,m}$, such that $\mu_1>\mu_2>\cdots>\mu_m>0$. For every $t$, the last box of the last row with $\mu_t$ elements can be removed to obtain a partition $\rho \lessdot \mu$. The content of that box is $\mu_t - \sum_{s=1}^t k_s$ by~\eqref{eq:defcontent}, such that
	\begin{equation}\label{eq:rhs}
		\sum_{\rho \lessdot \mu} c(\mu/\rho)
		= \sum_{t=1}^{m} \left(\mu_t - \sum_{s=1}^{t} k_s  \right)
	\end{equation}
	since these $m$ boxes are the only boxes that can be removed to yield a $\rho \lessdot \mu$. Similarly, the left-hand side of~\eqref{eq:technicalidentity} has $m+1$ terms. Namely, for every $t=1,2,\dots,m$, a box can be added to the first row of length $\mu_t$, and also a new row can be created at the end. Therefore, 
	\begin{equation}\label{eq:lhs}
		\sum_{\lambda \gtrdot \mu} c(\lambda/\mu)
		= \sum_{t=1}^{m} \left((\mu_t+1) - \left(1+ \sum_{s=1}^{t-1} k_s\right) \right) +1 -\left(1+\sum_{s=1}^{t} k_s\right)
	\end{equation}
	which follows directly from working out the contents of the boxes that are involved. Finally, it is straightforward to observe that the right-hand sides of~\eqref{eq:rhs} and~\eqref{eq:lhs} coincide and hence~\eqref{eq:technicalidentity} is established.
\end{proof}

\begin{proposition}\label{prop:technicalidentity}
	Let $\mu$ be a partition. Then
	\begin{equation}\label{eq:corollary}
		\sum_{\lambda \gtrdot \mu} F_{\lambda} \, c(\lambda/\mu) 
		= 0
	\end{equation}
	where $c(\lambda/\mu)$ is defined by~\eqref{eq:defcontent}.
\end{proposition}
\begin{proof}
	The proof is by induction on $n:=|\mu|$. If $n=0$, then $\mu=\emptyset$ and the sum in~\eqref{eq:corollary} only consists of the term $\lambda=(1)$ and hence the result is trivial. Therefore, let $n>0$ and assume that the result holds for all partitions of size $k<n$. Note that
	\begin{equation*}
		\sum_{\lambda \gtrdot \mu} F_{\lambda} \, c(\lambda/\mu) 
		= \sum_{\lambda \gtrdot \mu} \sum_{\gamma \lessdot \lambda} F_{\gamma} \, c(\lambda/\mu)
	\end{equation*}
	since $F_\lambda = \sum_{\gamma \lessdot \lambda} F_\gamma$. The claim is now that
	\begin{equation}\label{eq:proofcorollary1}
		\sum_{\lambda \gtrdot \mu} \sum_{\gamma \lessdot \lambda} F_{\gamma} \, c(\lambda/\mu)
		= \sum_{\rho \lessdot \mu} \sum_{\gamma \gtrdot \rho}  F_{\gamma} \, c(\gamma/\rho)
	\end{equation}
	and hence, by applying the induction hypothesis for every $\rho \lessdot \mu$ in the last sum, the right-hand side vanishes. This then establishes that~\eqref{eq:corollary} holds. To prove~\eqref{eq:proofcorollary1}, define for any partition $\mu$ the set
	\begin{equation*}
		S(\mu)
		:= \{\gamma \vdash |\mu| \mid \exists \lambda \gtrdot \mu \text{ such that } \gamma \lessdot \lambda \text{ and } \gamma \neq \mu\}
	\end{equation*}
	and for any two partitions $\gamma$ and $\nu$ define the partition $\gamma\wedge\nu$ (respectively $\gamma\vee\nu$) identified with the intersection (respectively union) of both Young diagrams of~$\gamma$ and~$\nu$. Then the left-hand side of~\eqref{eq:proofcorollary1} is
	\begin{equation}\label{eq:proofcorollary2}
		\sum_{\gamma\in S(\mu)} F_\gamma \, c\big((\gamma\vee\mu)/\mu\big) + F_{\mu} \sum_{\lambda \gtrdot \mu} \, c(\lambda/\mu)
	\end{equation}
	whereas the right-hand side equals
	\begin{equation}\label{eq:proofcorollary3}
		\sum_{\gamma\in S(\mu)} F_\gamma \, c\big(\gamma/(\mu\wedge\gamma)\big) + F_{\mu} \sum_{\rho \lessdot \mu} \, c(\mu/\rho) 
	\end{equation}
	because $S(\mu)$ can alternatively be written as
	\[S(\mu)= \{\gamma \vdash |\mu| \mid \exists \rho \lessdot \mu \text{ such that } \gamma \gtrdot \rho \text{ and } \gamma \neq \mu\}\]
	since the Young lattice is a 1-differential poset~\cite{Stanley_DiffPos}. By Lemma~\ref{lem:technicalidentity}, the last sum in expressions~\eqref{eq:proofcorollary2} and~\eqref{eq:proofcorollary3} coincide. Moreover, for any $\gamma \in S(\mu)$, the skew partitions $(\gamma\vee\mu)/\mu$ and $\gamma/(\mu\wedge\gamma)$ are the same, and hence the first sums in~\eqref{eq:proofcorollary2} and~\eqref{eq:proofcorollary3} are equal term by term. This establishes identity~\eqref{eq:proofcorollary1} and therefore ends the proof.
\end{proof}

\begin{remark}
	The identity in Proposition~\ref{prop:technicalidentity} can be expressed in terms of the degree vector via~\eqref{eq:Flambda} and $\lambda_i=n_i-r+i$ for $i=1,2,\dots,r$. Simplifying the equation yields the identity
	\begin{equation*}
		\sum_{k=1}^{r} \frac{n_k+1-r}{n_k+1}\prod_{j\neq k} \frac{n_k+1-n_j}{n_k-n_j}
		= r \prod_{k=1}^{r}\frac{n_k}{n_k+1}
	\end{equation*}
	which can be proven by induction on $r$, i.e., the number of elements, and using the same ideas as the proof of Lemma 8 in~\cite{Bonneux_Kuijlaars}. This gives an alternative way to prove identity~\eqref{eq:corollary}.
\end{remark}

\section*{Acknowledgements}
The authors thank Zachary Hamaker and John Stembridge for valuable discussions and Arno Kuijlaars for carefully reading the manuscript. They are supported in part by the long term structural funding-Methusalem grant of the Flemish Government, and by EOS project 30889451 of the Flemish Science Foundation (FWO). Marco Stevens is also supported by the Belgian Interuniversity Attraction Pole P07/18, and by FWO research grant G.0864.16.

\end{document}